\documentclass[11pt]{amsart}

\usepackage[all]{xy}
\usepackage{fullpage, amsfonts, amsmath, amsthm, epsfig, paralist, psfrag, subfig, setspace}
\usepackage{amssymb, url, enumerate, mathrsfs}
\usepackage{amscd}
\usepackage{url}

\usepackage{hyperref}
\usepackage{mathtools}

\usepackage{tikz}
\usepackage{tikz-cd}
\usetikzlibrary{positioning}

\def\id{\operatorname{id}}

\def\Hom{\operatorname{Hom}}

\def\C{\mathbf k}

\def\uC{\mathcal{C}}
\def\uV{\mathcal{V}}

\def\uM{\mathcal{M}}
\def\uG{\mathcal{G}}
\def\uH{\mathcal{H}}
\def\uR{\mathcal{R}}

\def\uN{\mathcal{N}}

\def\vctk{\mathrm{Vect}_{\C}}

\def\ob{\mathrm{ob}}
\def\OneHom{\mathrm{1\mbox{-}Hom}}
\def\TwoHom{\mbox{2-Hom}}

\def\vgw{\mathrm{Vec}_G^\alpha}

\def\ind{\mathrm{ind}}

\usepackage{fouridx}

\def\Tr{\mathbb{T}\mathrm{r}}

\def\twotrunc{\Pi_1}

\newcommand{\longmap}[3]{#1\negmedspace : #2 \longrightarrow #3}
\newcommand{\inv}{^{-1}}
\newcommand{\ZZ}{\mathbb Z}

\newtheorem{theorem}{Theorem}[section]
\newtheorem*{theorem*}{Theorem}
\newtheorem*{maintheorem*}{Main Theorem}
\newtheorem{proposition}[theorem]{Proposition}
\newtheorem{corollary}[theorem]{Corollary}
\newtheorem*{pcorollary*}{Corollary of proof}
\newtheorem{lemma}[theorem]{Lemma}

\theoremstyle{definition}
\newtheorem{definition}[theorem]{Definition}
\newtheorem{example}[theorem]{Example}
\newtheorem{remark}[theorem]{Remark}

\usepackage{titletoc}

\usepackage{graphicx}
\graphicspath{ {./images/} }
\DeclareGraphicsExtensions{.pdf,.png,.jpg}

\begin{document}
\title{Representation and character theory of finite categorical
  groups}
\author{Nora Ganter}
\author{Robert Usher}
\address{Department of Mathematics and Statistics\\
The University of Melbourne\\ Parkville VIC 3010\\ Australia}
\thanks{This paper is based on Usher's Masters Thesis, completed in
  2013 at the University of 
  Melbourne. Ganter was supported by an Australian Research Fellowship
  and by ARC grant DP1095815.}
\begin{abstract}
  We study the gerbal representations of a finite
  group $G$ or, equivalently, 
  module categories over Ostrik's category $\mathrm{Vec}_G^\alpha$ for
  a 3-cocycle $\alpha$. We adapt
  Bartlett's string diagram formalism to this situation to prove
  that the categorical character of a gerbal representation is 
  a representation of the inertia groupoid of a categorical group.
  We interpret such a representation as a module over the twisted Drinfeld
  double $D^\alpha(G)$.
\end{abstract}

\maketitle
\tableofcontents

\section{Introduction}
Let $\C$ be a field.
In classical representation theory, there are several equivalent 
definitions of the notion of a projective representation of a finite
group $G$ on a $\C$-vector space $V$:
\begin{enumerate}[(i)]
	\setlength{\itemsep}{9pt}
\item a group homomorphism $\varrho : G \longrightarrow \mathrm{PGL}(V)$,
\item a map $\varrho : G \longrightarrow \mathrm{GL}(V)$ with
  2-cocycle $\theta : G \times G \longrightarrow \C^{\times}$ such
  that 
\[
\varrho(g)\cdot \varrho(h) = \theta(g,h)\cdot \varrho(gh)
\]
\item a group homomorphism $\varrho : \widetilde{G} \longrightarrow
  \mathrm{GL}(V)$, for $\widetilde{G}$ a central extension of $G$ by
  $\C^{\times}$, 
\item a module over the twisted group algebra $\C^\theta[G]$ for some
  2-cocycle ${\theta : G \times G \longrightarrow \C^{\times}}$. 
\end{enumerate}
In this work we consider the situation where $V$ is replaced by a
$\C$-linear category $\mathcal V$ or, more generally, by an object of
a $\C$-linear strict 2-category.
In \cite{Frenkel:Gerbal_representations}, Frenkel and Zhu categorified
points (i) to (iii) as follows\footnote{We have slightly modified the
  context of their definitions to suit our purposes, demanding
  $\C$-linearity, while allowing ourselves to work in the
  2-categorical setup.}
\begin{enumerate}[(i)]
	\setlength{\itemsep}{9pt}
\item a homomorphism of groups $G \longrightarrow
  \pi_0(\mathrm{GL}(\mathcal V))$, see \cite[Definition 2.8]{Frenkel:Gerbal_representations},
\item a projective 2-representation of $G$ on $\uV$ for some 3-cocycle $\alpha
  : G \times G \times G \longrightarrow \C^{\times}$, see  \cite[Remark
  2.9]{Frenkel:Gerbal_representations}, 
\item a homomorphism of categorical groups $\uG \longrightarrow
  \mbox{GL}(\mathcal V)$ where $\uG$ is a 2-group extension of $G$
  by $[\mathrm{pt}/\C^{\times}]$, see \cite[Definition 2.6]{Frenkel:Gerbal_representations}.
\end{enumerate}
They prove that these three notions specify the same data and coin the term
{\em gerbal representation} of $G$ to describe any of these categorifications (we will
also use the term {\em projective 2-representation}).  To be more precise,
the objects described in (ii) and (iii) are in an obvious manner organised into
bicategories, and the argument in \cite[Theorem 2.10]{Frenkel:Gerbal_representations} sketches
an equivalence of bicategories.  The objects in (i) classify the objects of either of these
bicategories up to equivalence.  We review the work of Frenkel and Zhu in Section
\ref{sec:projective_2-representations}. A special case of \cite{Ostrik:Module_categories_Drinfeld} yields a
categorification of the last point:
\begin{enumerate}[(i)]
\setcounter{enumi}{3}  
\item a module category over the categorified twisted group algebra
  $\operatorname{Vect_\C^\alpha[G]}$ or, in Ostrik's notation, $\vgw$.
\end{enumerate}
There is an equivalence of bicategories between this formulation and those of (i)--(iii),
see Section \ref{sec:2-representations_as_module_categories}.  

An important class of examples of projective 2-representations are braid group actions,  which
can be read about in a paper of Khovanov and Thomas \cite{KT:BraidTC} building on work of Deligne \cite{Deligne:AdG}.
Projective 2-representations also also expected to play a role in TQFT applications; in \cite{FHLT:TQFT}
the authors argue that the categorified twisted group algebra determines a 3-dimensional extended TQFT
whose value at the point is $\operatorname{Vect_\C^\alpha[G]}$.

The goal of the present work is to describe the characters of projective 2-representations.  The special case
where $\alpha = 0$ was treated in \cite{GK:Representation_and_character_theory} and  
\cite{Bartlett:Unitary_2_representations}, where the character is defined using the \emph{categorical trace}
\[
X_\varrho(g) \,=\, \Tr(\varrho(g))
\,=\, \TwoHom(1_V, \varrho(g)).
\]
The \emph{categorical character} of $\varrho$ then consists of the $X_\varrho(g)$ together with a family of isomorphisms
\[
\beta_{g,h} : X(g) \longrightarrow X(hgh^{-1})
\]
(compare Definition \ref{def:character}).
We generalise these definitions to the projective case and arrive at the following theorem.

\begin{maintheorem*}
  Let $\uG$ be a finite categorical group, let $V$ be an object of a
 $\C$-linear strict 2-category and let $$\longmap\varrho{\uG}{\mathrm{GL}(V)}$$
  be a linear representation of $\uG$ on $V$. Then the categorical
  character of $\varrho$ is a representation of the inertia groupoid
 $\twotrunc \Lambda\uG$ of $\uG$.
(see Theorem \ref{thm:char;rep;twdd})
\end{maintheorem*}

Using the work of Willerton \cite{Willerton:Twisted_Drinfeld_double}, we further show Corollary \ref{cor:tigtwdd}:
The category of representations of $\twotrunc \Lambda\uG$ is equivalent to that of modules over the twisted
Drinfeld double $D^\omega(G)$.

\subsection{Acknowledgements} We would like to thank 
Simon Willerton for helpful conversations. Many thanks
go to Matthew Ando, who greatly helped to sort through some of the proofs.  We would also
like to thank the anonymous referee for many helpful suggestions.
\section{Background}
\subsection{Categorical groups}
By a {\em categorical group} or {\em 2-group} we mean a monoidal groupoid $(\uG, \bullet,
\textbf{1})$ where each object is weakly invertible.  For a detailed
introduction to the subject, we 
refer the reader to \cite{Baez:2-groups}, where the term {\em weak 2-group} is used.

\begin{example}[Symmetry 2-groups]\label{exa:symmetry_2-groups}
Let $\uV$ be a category.  Then the autoequivalences of $\uV$ and the
natural isomorphisms between them form a categorical group. 
More generally, let $V$ be an object in a bicategory.  Then the weakly
invertible 1-morphisms of $V$ and the 2-isomorphisms 
between them form the categorical group $\mathrm{1Aut(V)}$. If $V$ is a
$\C$-linear category, we may restrict ourselves to linear functors and natural transformations
and write $\mathrm{GL}(V)$.  We will also use this notation in general $\C$-linear 2-categories.
\end{example}

\begin{example}[Skeletal categorical groups]
Let $\uG$ be a skeletal 2-group, i.e., assume that each isomorphism
class in $\uG$ contains exactly one object.
Then the objects of $\uG$ form a group $G := \ob(\uG)$, and the
automorphisms of $\textbf{1}$ form an abelian group $A :=
\mathrm{aut}_{\uG}(\textbf{1})$.  The  
group $G$ acts on $A$ by conjugation
\[
a \longmapsto g \bullet a \bullet g^{-1}
\]
(unambiguous, because $\uG$ is skeletal).  We will denote this action by
\[
a^{g} := g \bullet a \bullet g^{-1}.
\]
We make the assumption that $\uG$ is {\em special}, i.e. that the
unit isomorphisms are identities, i.e., 
\[
\textbf{1} \bullet g = g = g \bullet \textbf{1}.
\]
Then $\uG$ is completely determined by the data above together with the 3-cocycle
\[
\alpha : G \times G \times G \longrightarrow A, 
\]
encoding the associators
\[
\alpha(g,h,k) \bullet ghk \in \mathrm{aut}_{\uG}(ghk).
\]
\label{ex:ske2grp}
\end{example}
Every finite categorical group is equivalent to one of this form, and
there is the following result of Sinh.
\begin{theorem}[{see \cite{Sinh:Gr_categories} and \cite[\S 8.3]{Baez:2-groups}}]
Let $\uG$ be a finite categorical group.  Then $\uG$ is determined up
to equivalence by the data of 
\begin{enumerate}[(i)]
\item a group $G$,
\item a $G$-module $A$, and
\item an element $[\alpha]$ of the group cohomology $H^3(G,A)$.
\end{enumerate}
\label{thm:sinh}
\end{theorem}
Without loss of generality, we may assume the cocyle $\alpha$ to be
normalised. 
We will be particularly interested in the case where $A = \C^\times$.
Cocycles of this form are key to
our understanding of a variety of different topics, ranging from
Chern-Simons theory (\cite{Dijkgraff:Topological_gauge_theories}, \cite{Freed:Chern-Simons_theory}) to
generalised and Mathieu moonshine (\cite{Ganter:Hecke_operators}, \cite{Gaberdiel:2012gf})
to line bundles on Moduli spaces and twisted 
sectors of vertex operator algebras.   
In the physics literature, evidence of such cocycles typically turns
up in the form of so called {\em phase factors}.

\begin{definition}
  Let $\uG$ and $\uH$ be categorical groups.
  By a {\em homomorphism} from $\uG$ to $\uH$ we mean a (strong)
  monoidal functor
  \[
    \uG \longrightarrow \uH.
  \]
  A {\em linear representation} of a categorical group $\uG$ with
  centre $\mathrm{End}_{\uG}(\textbf{1})=\C^\times$ is a homomorphism
  \[
    \varrho : \uG \longrightarrow \mathrm{GL}(V)
  \]
  where
  $V$ is an object of a strict ($\C$-linear) 2-category, and $\C^\times$ is
  required to act by multiplication with scalars.  
\label{def:2rep} 
\end{definition}
We will study such linear representations for skeletal $\uG$. Note
that the condition on the action of $\C^\times$ implies that the
action of $G=\ob(\uG)$ on $\C^\times$ is trivial, restricting us to
those skeletal 2-groups that are classified by  
$
 [\alpha]\in H^3(G;\C^\times)
$
where $G$ acts trivially on $\C^\times$.
\subsection{String diagrams for strict 2-categories}
We recall the string diagram formalism from \cite[\S 1.1]{Caldararu:The_Mukai_pairing} and
\cite[Chapter 4]{Bartlett:Unitary_2_representations} (our diagrams are upside down  
in comparison to those in \cite{Bartlett:Unitary_2_representations}).   
This already turns up in \cite{Penrose:AoNDT} and in the work of Joyal and Street \cite{AS:GTCI},
with a reference to \cite{KL:Coherence}.

Let $\uC$ be a strict 2-category, i.e. a category enriched over the
category of small categories.  Let $x, y$ be objects in $\uC$, and let
$A$ be a 1-morphism from $x$ to $y$.  In string diagram notation, $A$
is drawn 
\begin{center}
\includegraphics{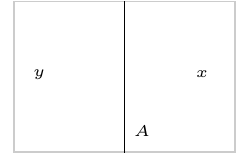}
\end{center}
Given $A, B \in \OneHom_{\uC}(x,y)$, let $\phi : A \Rightarrow B$ be a
2-morphism.  In string diagram notation, $\phi$ is drawn 
\begin{center}
\includegraphics{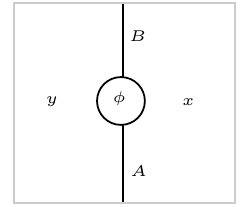}
\end{center}
So, our string diagrams are read from right to left and from bottom to top.  
Horizontal and vertical composition are represented by the respective concatenations of string diagrams.  For example, if $A \in \OneHom_{\uC}(x,y)$, and $\Phi : B \Rightarrow B'$ is a 2-morphism
between $B, B' \in \OneHom_{\uC}(y,z)$, then the horizontal composition of $A$ with $\Phi$ is represented by either of the diagrams
\begin{center}
\includegraphics{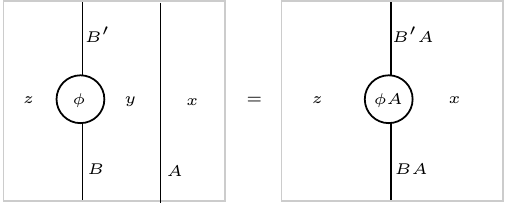}
\end{center}
The equals sign in this figure indicates that both string diagrams refer to the same 2-morphism.  
Given $C, D \in \OneHom_{\uC}(y,z)$, let $\psi : C \Rightarrow D$ be a 2-morphism, then the horizontal composition of $\phi$ with $\psi$ is represented
by
\begin{center}
\includegraphics{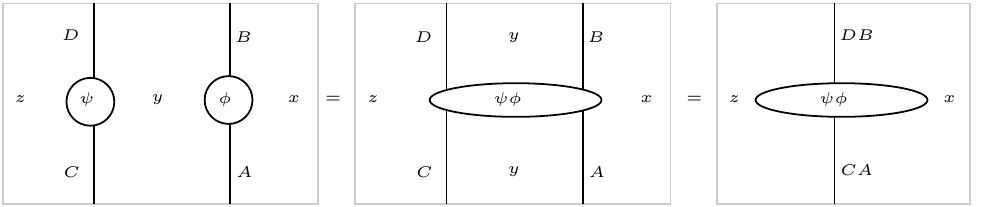}
\end{center}{\ \\}
If $\phi : A \Rightarrow B$ and $\phi' : B \Rightarrow C$ are composable 2-morphisms, their vertical composition is represented by
\begin{center}
\includegraphics{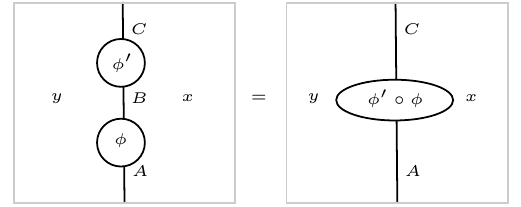}
\end{center}
We will often work with $\C$-linear 2-categories.  There each $\TwoHom_{\uC}(A,B)$ is a $\C$-vector space and vertical composition is $\C$-bilinear.  If $\phi, \phi \in \TwoHom_{\uC}(A,B)$ are 2-morphisms related
by a scalar $s \in \C$ (i.e. $s \phi = \phi'$), then we draw
\begin{center}
\includegraphics{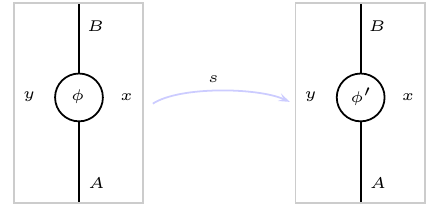}
\end{center}
We will occasionally omit borders and labels of diagrams where the context is clear.

\begin{remark}
Given a categorical group $\uG$, one could use a strictification result, as suggested in \cite{Bartlett:Unitary_2_representations},
to make use of string diagram notation.  Rather than doing this, we note that string diagrams for skeletal categorical groups are also unambiguous. The string diagrams appearing in the next section are similar to the ones above, but differ
in that we now need to keep track of associators.
\end{remark}

\section{Categorified conjugacy classes}\label{sec:inertia}
The goal of this section is to describe the inertia groupoid of a skeletal categorical group
and interpret modules over $D^\alpha(G)$ as
representations of that inertia groupoid.
\subsection{Homomorphisms of skeletal categorical groups}
A categorical group $\uG$ may be viewed as one-object bicategory. We
will denote this bicategory $\bullet /\!\!/ \uG$, i.e., the object is
$\bullet$ and
$\OneHom(\bullet,\bullet) =\uG$.
In this section we study the bicategory of bifunctors
$$
  \mathrm{Bicat}(\bullet/\!\!/ \uH,\bullet/\!\!/\uG),
$$
where $\uH$ and $\uG$ are skeletal categorical groups, classified by
cocycles 
\begin{eqnarray*}
  \alpha\negmedspace : G\times G\times G & \longrightarrow & A,
\end{eqnarray*}
and 
\begin{eqnarray*}
  \beta\negmedspace : H\times H\times H & \longrightarrow & B,
\end{eqnarray*} 
as in Example \ref{ex:ske2grp}. 
One may think of this as the bicategory of representations of $\uH$ in
$\uG$, but this time the target is not a strict 2-category. 

\subsubsection{The Objects}
Objects of $\mathrm{Bicat}(\bullet/\!\!/ \uH,\bullet/\!\!/\uG)$ are
homomorphisms of categorical groups, a.k.a. strong
monoidal functors, from $\uH$ to $\uG$. Such a homomorphism
is determined by the following data:
a group homomorphism $\longmap\varrho HG$, an $H$-equivariant homomorphism
$\longmap fBA$, and a 2-cochain $\gamma : H \times H \longrightarrow A$.  For $h_1, h_2 \in H$,
we draw $\gamma(h_1,h_2) \bullet \varrho(h_1 h_2)$ as

\begin{center}
  \includegraphics{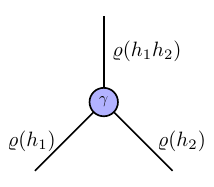}
\end{center}
If we let $H$ act on $A$ via $\varrho$ 
then $\gamma$ has to satisfy
\begin{eqnarray*}
  d\gamma & = & \frac{\varrho^*\alpha}{f_*\beta},
\end{eqnarray*}
i.e., 
\begin{eqnarray}
\label{eq:dgamma}
  \frac{\gamma(h_1h_2,h_3)\cdot\gamma(h_1,h_2)}{\gamma(h_1,h_2h_3)\cdot\gamma(h_2,h_3)^{\varrho(h_1)}}
  &\,\, = \,\,& \frac{\alpha(\varrho(h_1),\varrho(h_2),\varrho(h_3))}{f(\beta(h_1,h_2,h_3))}.
\end{eqnarray}
for all $h_1, h_2, h_3 \in H$. The hexagon equation \eqref{eq:dgamma} is drawn in string diagram notation as
\smallskip

\begin{center}
  \includegraphics{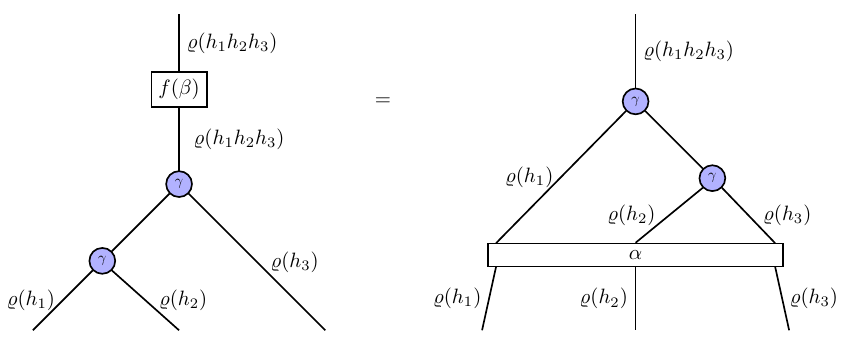}\\ 

\medskip
\end{center}

A priori, one expects one more piece of data, namely an arrow
\begin{eqnarray*}
  a\negmedspace : \varrho(1) & \longrightarrow & 1,
\end{eqnarray*}
i.e, an element $a\in A$, satisfying\footnote{Note that the axioms in
  \cite{Leinster:Basic_bicategories} are formulated in terms of $a\inv$.} 
\begin{eqnarray*}
  \gamma(1,h) & = & a\\
  \gamma(h,1) & = & a^{\varrho(h)}
\end{eqnarray*}
for all $h\in H$.
Since $\alpha$ and $\beta$ are normalised, 
this is automatic from
\eqref{eq:dgamma}. Indeed, set $a=\gamma(1,1)$ and apply
\eqref{eq:dgamma} to the triples $(1,1,h)$ and $(h,1,1)$.
\subsubsection{The 1-morphisms}
Let $(\varrho, f_1,\gamma_1)$ and $(\sigma,f_2,\gamma_2)$ be homomorphisms
from $\uH$ to $\uG$. Then the 1-morphisms between them are
transformations from $(\varrho, f_1,\gamma_1)$ to
$(\sigma,f_2,\gamma_2)$. We will follow the conventions in \cite{Gordon:Power:Street}.
A transformation then amounts to the data of an element $s \in G$, together with a 1-cochain $\longmap\eta HA$ satisfying
\begin{equation}
  d_\sigma\eta (h_1,h_2) \,\,:=\,\,
  \frac{\eta(h_2)^{\sigma(h_1)}\cdot\eta(h_1)}{\eta(h_1h_2)}\,\,=\,\,
  \frac{\gamma_1(h_1,h_2)^s}{\gamma_2(h_1,h_2)}
  \cdot
  \frac{\alpha(\sigma(h_1),\sigma(h_2),s)\cdot\alpha(s,\varrho(h_1),\varrho(h_2))}
  {\alpha(\sigma(h_1),s,\varrho(h_2))}.  
  \label{eq:deta}
\end{equation}
for all $h_1, h_2 \in H$.  For $h \in H$, we draw $\eta(h) \bullet \sigma(h) \bullet s$ as
\begin{center}
  \includegraphics{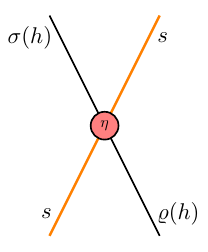}
\end{center}
and so the eight-term equation \eqref{eq:deta} is drawn in string diagram notation as
\begin{center}
  \includegraphics{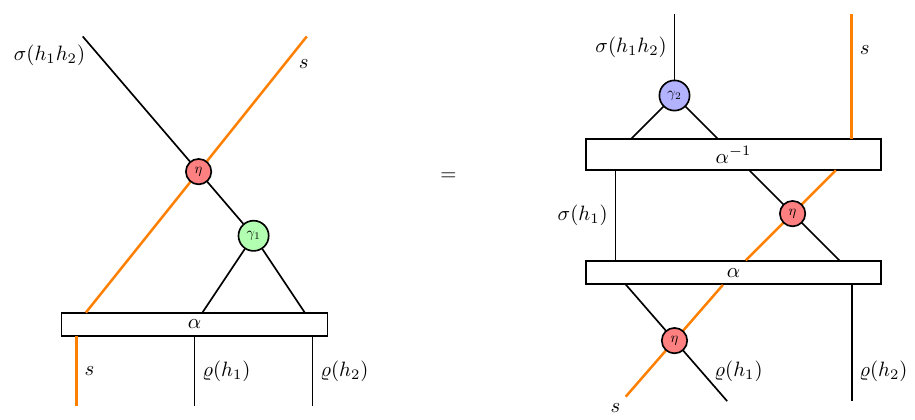}\\
\end{center}

The second condition spells out to
\begin{eqnarray*}
  \eta(1) & = & \frac{\gamma_1(1,1)^s}{\gamma_2(1,1)}
  \end{eqnarray*}
which we do not postulate, since in our situation it is automatic from
\eqref{eq:deta}. 
Indeed, it is obtained from the formula for $d\eta(1,1)$, because
$\alpha$ is normalised.

\subsubsection{The 2-morphisms}
A modification from $(s,\eta)$ to $(t,\zeta)$ requires $s=t$ and
amounts to a 0-cochain
$\omega$ (i.e. an element $\omega \in A$)
satisfying
\begin{eqnarray}\label{eq:domega}
  \frac{\omega^{\sigma(h)}}{\omega} & = & \frac{\zeta(h)}{\eta(h)} 
\end{eqnarray}
for all $h \in H$. We draw $\omega \bullet s$ as
\begin{center}
\includegraphics{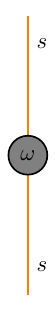}
\end{center}
and so condition \eqref{eq:domega} is drawn in string diagram notation as
\smallskip
\begin{center}
  \includegraphics{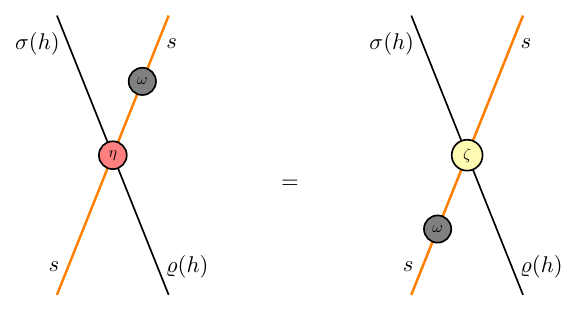}
\end{center}
\begin{example}[Group extensions]
Let $G$ be a group, and let $A$ be an abelian group. Then 
the bicategory of bifunctors from 
$\bullet/\!\!/G$ to $\bullet/\!\!/\bullet/\!\!/ A$
has as objects 2-cocycles on $G$ with values in $A$, viewed as a trivial
$G$-module. A 1-morphism from
$\gamma_1$ to $\gamma_2$ is a 1-cochain $\eta$ with $d\eta
=\gamma_2/\gamma_1$. All the 2-morphisms are 2-automorphisms, and each
2-automorphism group is isomorphic to $A$.  If we truncate at the level of 2-automorphisms,
then this is the category of central extensions of $G$ by $A$ and their isomorphisms over $G$.
\end{example}
\subsection{Inertia (2)groupoids}
\begin{definition}
  We define the inertia 2-groupoid of a categorical group $\uG$
  as the 2-groupoid
  \begin{eqnarray*}
    \Lambda\uG & = & \mathrm{Bicat} (\bullet/\!\!/\ZZ,\bullet/\!\!/\uG),
  \end{eqnarray*}
  where the integers are viewed as a discrete 2-group (only identity morphisms).  
\end{definition}
\begin{example}
  If $\uG=G$ is a finite group, viewed as a categorical group with only
  identity morphisms, then $\Lambda G$ is the usual inertia groupoid
  with objects $g\in G$ and arrows $g\to sgs\inv$.
\end{example}
In general, let $\uG$ be a special skeletal 2-group with objects
$\ob(\uG)=G$.
Then the canonical 2-group homomorphism
\begin{eqnarray*}
  p\negmedspace : \uG & \longrightarrow & G
\end{eqnarray*}
induces a morphism of 2-groupoids
\begin{eqnarray*}
  \Lambda p \negmedspace :\Lambda \uG & \longrightarrow &\Lambda G
\end{eqnarray*}
\begin{lemma}
  The map $\Lambda(p)$ is surjective on objects and full.
\end{lemma}
\begin{proof}
  The proof relies on the knowledge of the group cohomology of the
  integers, see for instance \cite[Exa. 3.1]{Brown:Lectures_on_the_cohomology}.
  The objects of $\Lambda\uG$ are identified with pairs $(g,\gamma)$,
  where $g$ is an element of $G$ (namely $g=\varrho(1)$) and 
  $$\longmap\gamma{\ZZ\times\ZZ} A$$ is a 2-cochain with boundary
  \begin{eqnarray*}
    d_g\gamma(l,m,n)\, :=\,
    \frac{\gamma(l+m,n)\cdot\gamma(l,m)}{\gamma(l,m+n)\cdot\gamma(m,n)^{g^l}} &
    = & \alpha(g^l,g^m,g^n).
  \end{eqnarray*}
  The map $\Lambda p$ sends $(g,\gamma)$ to $g$. 
  Since
  \begin{eqnarray*}
    H^3(\ZZ,A) & = & 0     
  \end{eqnarray*}
  for any
  $\ZZ$-action on $A$, we may conclude that $\Lambda p$ is surjective
  on objects. 
  Let now $(g,\gamma)$ and $(f,\phi)$ be two objects of
  $\Lambda\uG$, and assume that we are given an arrow from $g$ to $f$
  in $\Lambda G$. Such an arrow amounts to an element $s$ of $G$ with 
  \begin{eqnarray*}
    sgs\inv & = & f.
  \end{eqnarray*}
  Applying the cocycle condition for $\alpha$ four times, namely
  \begin{eqnarray*}
    d\alpha(s,g^l,g^m,g^n) &=&0,\\
    d\alpha(f^l,s,g^m,g^n)&=&0,\\
    d\alpha(f^l,f^m,s,g^n) &=&0,\\
    d\alpha(f^l,f^m,f^n,s)&=&0,
  \end{eqnarray*}
  we obtain that the 2-cochain 
  \begin{eqnarray*}
    (m,n) & \longmapsto &     
    \frac{\phi(m,n)}{\gamma(m,n)^s}\cdot
    \frac{\alpha(f^m,s,g^n)}{\alpha(f^m,f^n,s)\cdot\alpha(s,g^m,g^n)} 
  \end{eqnarray*}
  is a 2-cocycle for the $\ZZ$-action on $A$ induced by $f$.
  Since
  \begin{eqnarray*}
    H^2(\ZZ,A) & = & 0,
  \end{eqnarray*}
  we may conclude that $\Lambda p$ is surjective
  on 1-morphisms. 
\end{proof}
Let $\mathbf G$ be a groupoid, and let $A$ be an abelian group. We
recall from \cite[p.17]{Willerton:Twisted_Drinfeld_double} how an $A$-valued
2-cocycle $\theta$ on $\mathbf G$ defines a central extension
$\widetilde{\mathbf G}$ of $\mathbf G$. The objects of
$\widetilde{\mathbf G}$ are the same as those of $\mathbf G$. The
arrows are
\begin{eqnarray*}
  \Hom_{\widetilde{\mathbf G}}(g,h) & = & A \times   \Hom_{{\mathbf G}}(g,h)
\end{eqnarray*}
with composition 
\begin{eqnarray*}
  (a_1,g_1) (a_2,g_2) & := & (\theta(g_1,g_2)a_1a_2, g_1g_2).
\end{eqnarray*}
Let $\uG$ be the 2-group defined by $\alpha$ as above, and assume
that the $G$-action on $A$ is trivial.  Then all the 2-morphisms in
$\Lambda\uG$ are 2-automorphisms. In this case, we may view
$\Lambda\uG$ as a groupoid, forgetting the 2-morphisms.  Let us denote
by $\twotrunc \Lambda\uG$ the groupoid obtained by truncating $\Lambda\uG$ to forget 2-arrows.
\begin{proposition}\label{prop:transgression_cext}
  The groupoid $\twotrunc \Lambda\uG$ 
  is equivalent to the central extension of $\uG$ defined by 
  the transgression of $\alpha$,
  \begin{eqnarray*}
    {\tau(\alpha)\left(g\xrightarrow{\,\,s\,\,}h\xrightarrow{\,\,t\,\,}k\right)} &
    = &  \frac{\alpha(t,s,g)\cdot \alpha(k,t,s)}{\alpha(t,h,s)}.
  \end{eqnarray*}
\end{proposition}
\begin{proof}
For each $g\in G$, fix an object $(g,\gamma)$ of $\twotrunc \Lambda\uG$ mapping
to $g$ under $\Lambda p$. Since $\Lambda p$ is surjective on arrows,
the full subgroupoid $\twotrunc \Lambda\uG'$ of $\twotrunc \Lambda\uG$ with the objects we
just fixed is equivalent to $\twotrunc \Lambda\uG$.
Since $G$ (and hence $\ZZ$) acts trivially on $A$, the $A$-valued
one-cocycles on $\ZZ$ are just group homomorphisms from $\ZZ$ to
$A$. Hence, for any 2-cocycle $\xi\in Z^2(\ZZ,A)$, we have a bijection
\begin{eqnarray*}
  \{\eta \mid d\eta=\xi\} & \longrightarrow & A\\
   \eta & \longmapsto & \eta(1).
\end{eqnarray*}
Let now $s$ be an element of $G$, and let
\begin{eqnarray*}
  h=sgs\inv.
\end{eqnarray*}
Inserting the right-hand side of \eqref{eq:deta} for $\xi$, allows us
identify the set of arrows in $\twotrunc \Lambda\uG'$ mapping to $s$ with
$A$. Let now $t$ be another element of $G$ and let $k=tht\inv$. The
following string diagram illustrates the composition of arrows
$(s,\eta)$ and $(t,\zeta)$ in $\twotrunc \Lambda\uG'$.  
  \begin{center}
  \includegraphics{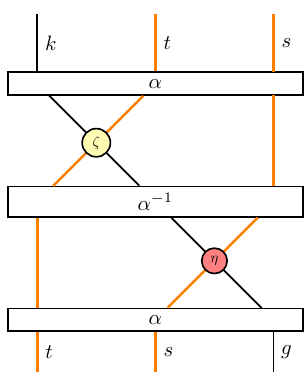}
  \end{center}
\end{proof}

\section{Projective 2-representations}\label{sec:projective_2-representations}
The following is a $\C$-linear version of \cite[Definition 2.8]{Frenkel:Gerbal_representations}.
\begin{definition}
Let $G$ be a finite group, and $\uC$ a strict $\C$-linear 2-category.  A
\emph{projective 2-representation} of $G$ on $\uC$ consists of the
following data 
\begin{enumerate}[(a)]

\item an object $V$ of $\uC$

\item for each $g \in G$, a 1-automorphism $\varrho(g) : V \longrightarrow V$, drawn as

\begin{center}
\includegraphics{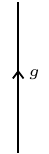}
\end{center}

\item for every pair $g, h \in G$, a 2-isomorphism $\psi_{g,h} : \varrho(g)\varrho(h) \stackrel{\cong}{\Rightarrow} \varrho(gh)$, drawn as

\begin{center}
\includegraphics{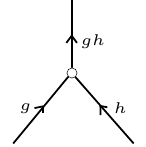}
\end{center}

\item a 2-isomorphism $\psi_1 : \varrho(1) \stackrel{\cong}{\Rightarrow} \id_{V}$, drawn as

\begin{center}
\includegraphics{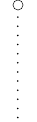}
\end{center}
\end{enumerate}

such that the following conditions hold

\begin{enumerate}[(i)]
\item for any $g, h, k \in G$, we have
\[
\psi_{g,hk}(\varrho(g) \psi_{h,k}) = \alpha(g,h,k)  \psi_{gh,k} (\psi_{g,h} \varrho(k)),
\]
where $\alpha(g,h,k) \in \C^\times$.  In string diagram notation, we draw this as
\begin{center}
\includegraphics{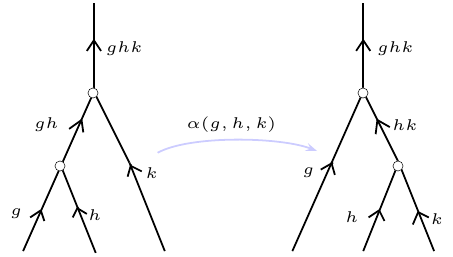}
\end{center}
\item for any $g \in G$, we have
\[
\psi_{1,g} = \psi_1 \varrho(g)  \mbox{ and }\psi_{g,1} = \varrho(g) \psi_1.
\]
In string diagram notation, we draw these as
\begin{center}{\ \\}
\begin{minipage}[c]{0.3\textwidth}
\includegraphics{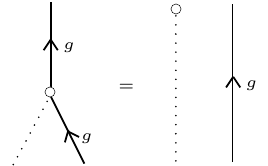}
\end{minipage}
\begin{minipage}[c]{0.08\textwidth}
 and 
\end{minipage}
\begin{minipage}[c]{0.3\textwidth}
\includegraphics{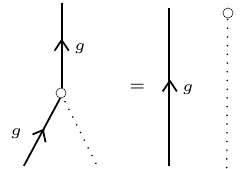}
\end{minipage}
\end{center}
\end{enumerate}
\label{def:proj2rep}
\end{definition}

\subsection{The 3-cocycle condition}

\begin{proposition}[{Compare \cite[Theorem 2.10]{Frenkel:Gerbal_representations}}]
Let $\varrho$ be a projective 2-representation of a group $G$.  Then
the map $\alpha : G \times G \times G \longrightarrow \C^\times$
appearing in condition (i) is a normalised 3-cocycle for the trivial
$G$-action on $\C^\times$.
\end{proposition}

\begin{proof}
We use Definition \ref{def:proj2rep} (i) for all steps of our proof.  Consider

\begin{center}
\includegraphics{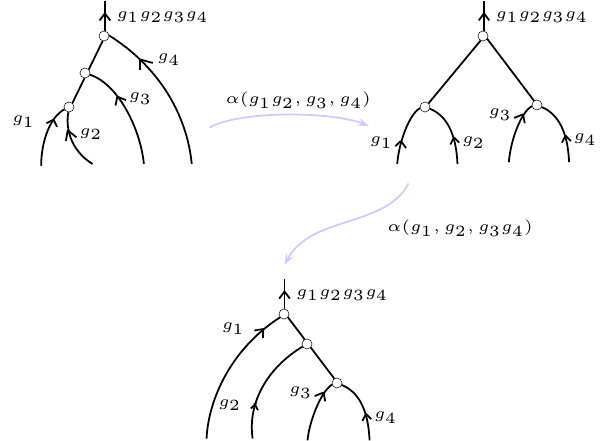}
\end{center}

On the other hand, we have 

\begin{center}
\includegraphics{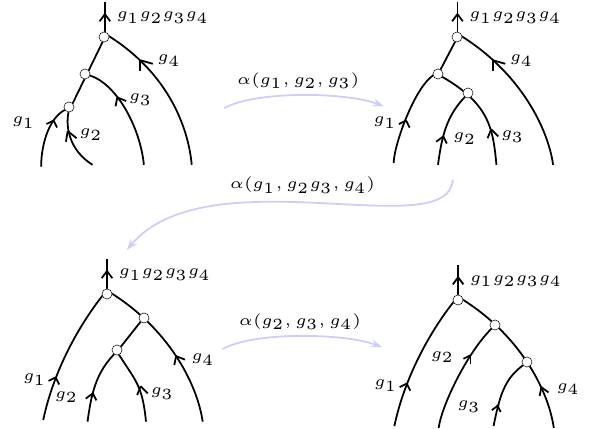}
\end{center}

Comparing these diagrams, we find
\begin{equation*}
\alpha(g_1g_2,g_3,g_4) \cdot\alpha(g_1,g_2,g_3g_4) =
\alpha(g_2,g_3,g_4)\cdot \alpha(g_1,g_2g_3,g_4)\cdot
\alpha(g_1,g_2,g_3), 
\end{equation*}
so $\alpha$ is indeed a 3-cocycle.
\end{proof}
\begin{corollary}
  A projective 2-representation of $G$ with cocycle $\alpha$ is precisely
  a linear representation of the 2-group $\mathcal G$
  classified by $\alpha$ (see Definition \ref{def:2rep}).
\end{corollary}
\begin{proof}
  Indeed, Condition (i) of Definition \ref{def:proj2rep} amounts to the 
hexagon diagram for a strong monoidal functor, and Condition (ii)
translates to the unit diagrams.   
\end{proof}
\begin{example}[{Compare \cite[\S 5.1]{GK:Representation_and_character_theory}}]
Let $G$ be a finite group, and let $\theta$ be a normalised 2-cochain.
Let $\alpha = d\theta$ be the coboundary of $\theta$, i.e., 
\[
\alpha(g,h,k) = \frac{\theta(gh,k)\cdot \theta(g,h)}{\theta(g,hk)\cdot
  \theta(h,k)}.
\]
Let $\mathrm{Vect_\C}$ be the category of finite dimensional
$\C$-vector spaces.
Then we define a projective 2-representation of $G$ on
$\mathrm{Vect}_{\C}$ with corresponding 3-cocycle $\alpha$ as follows :
for $g \in G$, we let 
\[
\varrho(g) = \id : \mathrm{Vect}_{\C} \longrightarrow \mathrm{Vect}_{\C}
\]
be the identity functor on $\mathrm{Vect}_{\C}$.  For $g, h \in G$
we let $$\psi_{g,h} : \id\circ\id \stackrel{\cong}{\Longrightarrow}
\id$$ 
be given by multiplication by $\theta(g,h)$. Further, $$\psi_1 : \varrho(1)
\stackrel{\cong}{\Longrightarrow} \id$$ is the identity natural transformation.
\label{ex:trivp2rep}
\end{example}
We recall some further notation from \cite{Bartlett:Unitary_2_representations}.
\smallskip

\begin{center}
\includegraphics{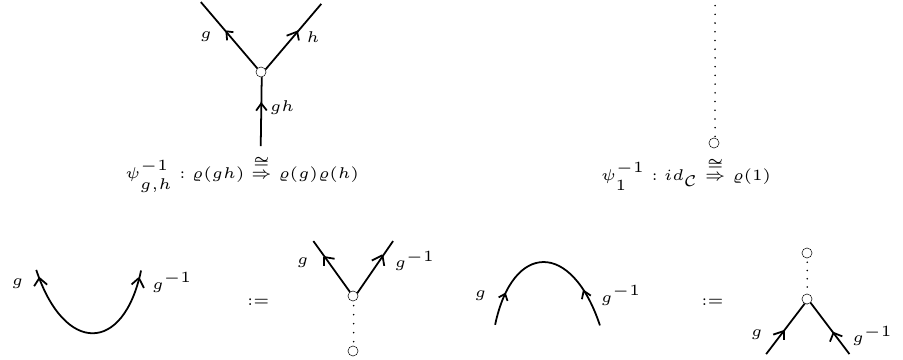}
\end{center}

\smallskip
\noindent

\begin{remark}
For future reference, we present the following tautological string diagram equations, as in \cite[\S 7.1.1]{Bartlett:Unitary_2_representations}.  By inverting condition (i) of Definition \ref{def:proj2rep}, we get
\smallskip

\begin{center}
\includegraphics{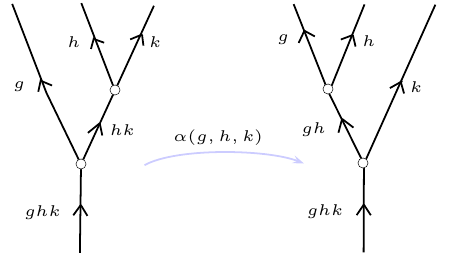}
\end{center}
that is,
\begin{equation}
  (\psi_{g,h}^{-1} \varrho(k)) \psi_{gh,k}^{-1} = \alpha(g,h,k)
  (\varrho(g) \psi_{h,k}^{-1}) \psi_{g,hk}^{-1} 
\label{cor:rstringa}
\end{equation}

We have $\psi_1 \circ \psi_1^{-1} = \id_{\uC}$ and $\psi_1^{-1} \circ \psi_1 = \varrho(1)$,
drawn as 

\begin{center}
\begin{minipage}[c]{0.05\textwidth}
$(a)$
\end{minipage}
\begin{minipage}[c]{0.25\textwidth}
\includegraphics{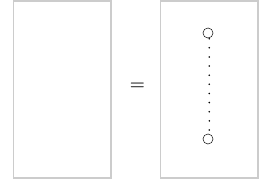}
\end{minipage}
\begin{minipage}[c]{0.1\textwidth}
\ 
\end{minipage}
\begin{minipage}[c]{0.05\textwidth}
$(b)$
\end{minipage}
\begin{minipage}[c]{0.25\textwidth}
\includegraphics{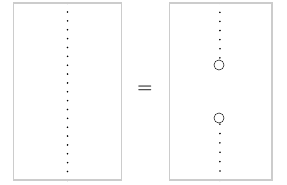}
\end{minipage}
\end{center}

Similarly, $\psi_{g,h}^{-1} \circ \psi_{g,h} = \varrho(g) \varrho(h)$ and $\psi_{g,h} \circ \psi_{g,h}^{-1} = \varrho(gh)$ for all $g, h \in G$, drawn as\\

\begin{center}
\begin{minipage}[c]{0.05\textwidth}
$(c)$
\end{minipage}
\begin{minipage}[c]{0.25\textwidth}
\includegraphics{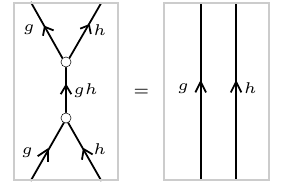}
\end{minipage}
\begin{minipage}[c]{0.1\textwidth}
\ 
\end{minipage}
\begin{minipage}[c]{0.05\textwidth}
$(d)$
\end{minipage}
\begin{minipage}[c]{0.25\textwidth}
\includegraphics{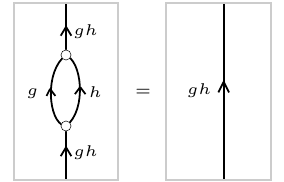}
\end{minipage}
\end{center}

Finally, 
$\psi_{1,g} (\psi_1^{-1} \varrho(g)) =\varrho(g) =  \psi_{g,1} (\varrho(g) \psi_1^{-1})$ for all $g \in G$, drawn as\\ 
\begin{center}
\begin{minipage}[c]{0.08\textwidth}
(e)
\end{minipage}
\begin{minipage}[c]{0.4\textwidth}
\includegraphics{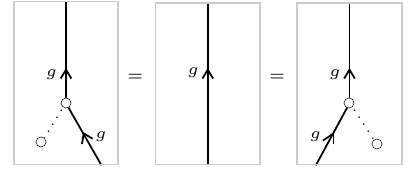}
\end{minipage}
\end{center}
%
\label{cor:gfxm}

\end{remark}

Some less tautological graphical equations for projective 2-representations are given by
the following results.
\begin{lemma}[{\cite[Lemma 7.3 (ii)]{Bartlett:Unitary_2_representations}}]
The following graphical equation holds

\smallskip
\begin{center}
\includegraphics{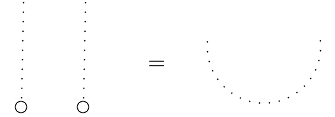}
\end{center}
\end{lemma}
\begin{proof}
\ 
\begin{center}
\includegraphics{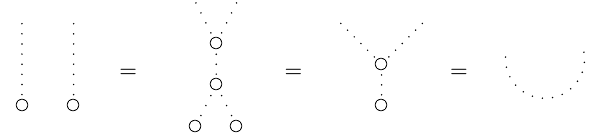}
\end{center}{\ \\}
The first equality follows from \ref{cor:gfxm} (c), the second from \ref{cor:gfxm} (e), with the final following by definition.
\end{proof}

\begin{lemma}[Compare {\cite[Lemma 7.3 (iii)]{Bartlett:Unitary_2_representations}}]
The following graphical equations hold\\

\begin{minipage}[c]{0.18\textwidth}
\hfill $(i)$\quad\quad
\end{minipage}
\begin{minipage}[c]{0.25\textwidth}
\includegraphics{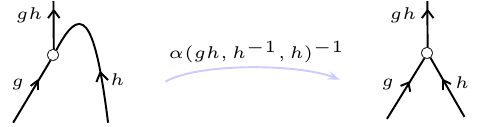}

\end{minipage}
{\ \\}\vspace{1cm}{\ \\}
\begin{minipage}[c]{0.18\textwidth}
\hfill $(ii)$\quad\quad
\end{minipage}
\begin{minipage}[c]{0.25\textwidth}
\includegraphics{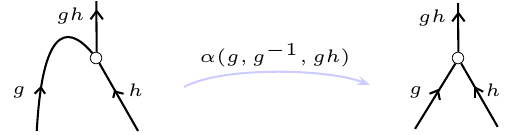}
\end{minipage}
\label{lem:bphd73iii}
\end{lemma}{\ \\}

\begin{proof}
We will prove $(ii)$; the proof of $(i)$ is almost identical.  By
combining \ref{cor:gfxm} (c) and (e), we obtain 
\begin{center}
\includegraphics{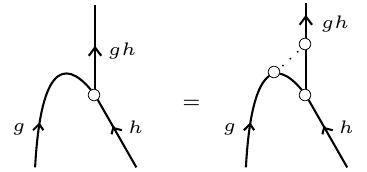}
\end{center}
Next, by \ref{def:proj2rep} (i), we obtain
\begin{center}
\includegraphics{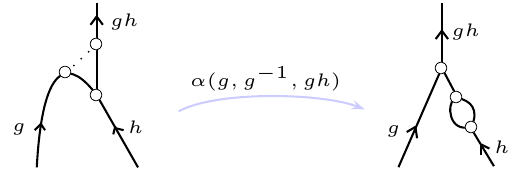}
\end{center}
A final application of \ref{cor:gfxm} (d) gives us the desired result.
\end{proof}

\begin{corollary}[Compare {\cite[Lemma 7.3 (i)]{Bartlett:Unitary_2_representations}}]
The following graphical equations hold\\

\begin{minipage}[c]{0.05\textwidth}
$(i)$
\end{minipage}
\begin{minipage}[c]{0.35\textwidth}
\includegraphics{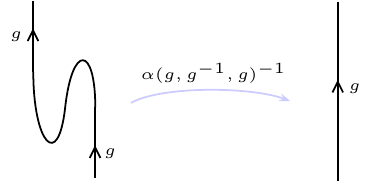}
\end{minipage}
\begin{minipage}[c]{0.1\textwidth}
\ 
\end{minipage}
\begin{minipage}[c]{0.05\textwidth}
$(ii)$
\end{minipage}
\begin{minipage}[c]{0.35\textwidth}
\includegraphics{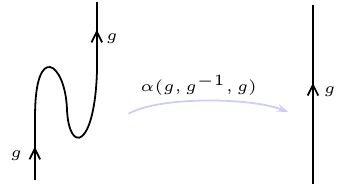}
\end{minipage}
\end{corollary}

\begin{proof}
We will prove $(i)$; the proof of $(ii)$ is almost identical.  By applying \ref{lem:bphd73iii} and then \ref{cor:gfxm} (e), we have

\begin{center}
\includegraphics{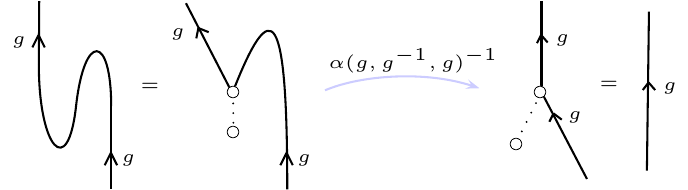}
\end{center}
as required.
\end{proof}

\begin{corollary}[Compare {\cite[Lemma 7.3 (iv)]{Bartlett:Unitary_2_representations}}] 
The following graphical equation holds\\
\begin{center}
\includegraphics{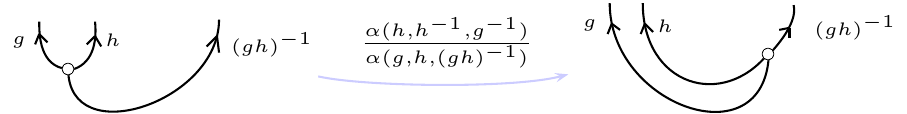}
\end{center}
\label{cor:73iv}
\end{corollary}

\begin{proof}
Applying \ref{cor:rstringa}. we get
\begin{center}
\includegraphics{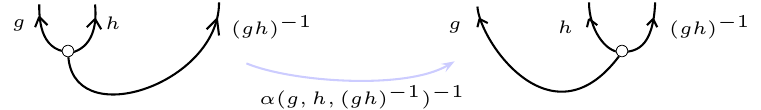}
\end{center}

Inverting the equation derived in part (ii) of \ref{lem:bphd73iii} gives the desired result.
\end{proof}
\subsection{The character of a projective 2-representation}
Recall that the character of a classical representation $\varrho$ is
the map $\chi : G \longrightarrow \C$ defined by $\chi(g) =
\mathrm{tr}(\varrho(g))$. 
This motivates the following definition of \cite{GK:Representation_and_character_theory} and \cite{Bartlett:Unitary_2_representations}.
\begin{definition}[{\cite[Definition 3.1]{GK:Representation_and_character_theory} and \cite[Definition 7.8]{Bartlett:Unitary_2_representations}}]
Let $\uC$ be a 2-category, $x \in \ob(\uC)$ and $A \in \OneHom_{\uC}(x,x)$ a 1-endomorphism of $x$.  The \emph{categorical trace} of $A$ is defined
to be
\[
\Tr(A) = \TwoHom_{\uC}(1_x,A)
\]
where $1_x$ is the identity 1-morphism of $x$.
\end{definition}

\begin{remark}
If $\uC$ is a $\C$-linear 2-category, then the categorical
trace of a 1-endomorphism ${A : x \longrightarrow x}$ is a $\C$-vector space.
\end{remark}

\begin{definition}[{Compare \cite[Definition 4.8]{GK:Representation_and_character_theory} and, in particular, \cite[Definition 7.9]{Bartlett:Unitary_2_representations}}]
Let $\varrho$ be a projective 2-representation of a finite group $G$.  The \emph{character} of $\varrho$ is the assignment
\[
g \longmapsto \Tr(\varrho(g)) =: X_\varrho(g)
\]
for each $g \in G$, and the collection of isomorphisms 
\[
\beta_{g,h} : X_\varrho(g) \longrightarrow X_\varrho(hgh^{-1})
\]
defined in terms of string diagrams

\begin{center}
\includegraphics{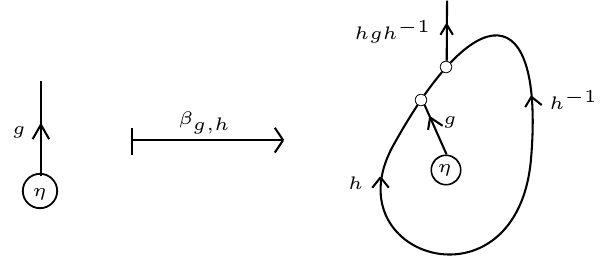}
\end{center}
for each $g, h \in G$.  That the $\beta_{g,h}$ are isomorphisms is a
consequence of Theorem \ref{thm:char;rep;twdd}.  
\label{def:character}
\end{definition}
We note that the
definitions in \cite{GK:Representation_and_character_theory} and \cite{Bartlett:Unitary_2_representations} are the special case
$\alpha = 1$, although they look a bit different at first sight. There
are several thinkable generalisations of those definitions. This definition was chosen
based on the discussion in Section \ref{sec:inertia}.
\begin{definition}[{\cite[Definition 4.12]{GK:Representation_and_character_theory}}]
  Let $\varrho$ be a projective 2-representation of a finite group $G$
  on a $\C$-linear 2-category. If $g, h \in G$ is a pair of commuting
  elements, then $\beta_{g,h}$ is an automorphism of $X_\varrho(g)$.  
  Assuming $\beta_{g,h}$ to be of trace class, we have the \emph{joint
    trace} of $g$ and $h$,
  \[
    \chi_\varrho(g,h) := tr(\beta_{g,h}).
  \]
  If the joint trace is defined for all commuting $g, h \in G$, we
  refer to $\chi_\varrho$ as the {\em 2-character} of $\varrho$.    
\end{definition}
\begin{example}
As in \cite[\S 5.1]{GK:Representation_and_character_theory}, let us consider the categorical character and 2-character
of the projective 2-representation defined in Example
\ref{ex:trivp2rep}.  For $g \in G$, we have 
\[
X_\varrho(g) = \Tr(\id_{\mathrm{Vec}_\C}) = \C
\]
Let $g, h \in G$ be commuting elements. Then it follows from
Definition \ref{def:character} that the joint trace $\chi_\varrho(g,h)$ is
given by multiplication by the transgression of $\theta$
\[
  \frac{\theta(h,g)}{\theta(hgh^{-1},h)} = \frac{\theta(h,g)}{\theta(g,h)}.  
\]
\end{example}

We now present our main result.  

\begin{theorem} \label{thm:char;rep;twdd} 
  Let $\uG$ be a finite categorical group, let $V$ be an object of a
 $\C$-linear strict 2-category and let $$\longmap\varrho{\uG}{\mathrm{GL}(V)}$$
  be a linear representation of $\uG$ on $V$. Then the categorical
  character of $\varrho$ is a representation of the inertia groupoid
 $\twotrunc \Lambda\uG$ of $\uG$.
\end{theorem}

\begin{proof}
Recall that $\uG$ is determined by a finite group $G$ together with a 3-cocycle $\alpha$ on $G$
with values in $\C^\times$.  By Proposition \ref{prop:transgression_cext}, we must show that the diagram

\begin{center}
\begin{tikzcd}
X_\varrho(r) \arrow{rr}{\beta_{r,hg}} \arrow{rd}{\beta_{r,g}} & & X_\varrho(hgrg^{-1}h^{-1})\\
& X_\varrho(grg^{-1}) \arrow{ru}{\beta_{grg^{-1},h}}
\end{tikzcd}
\end{center}
commutes up to a scalar, i.e.
\begin{equation}
\label{eq:transgression_alpha}
\frac{\alpha(h,grg^{-1},g)}{\alpha(hgrg^{-1}h^{-1},h,g)\cdot \alpha(h,g,r)}\cdot
\beta_{grg^{-1},h} \circ \beta_{r,g} = \beta_{r,hg}
\end{equation}
for all $r, g, h \in G$.  Fix elements $r, g,h \in G$ and $\eta \in X(r)$. By applying \ref{cor:gfxm} (d) twice, we find 
\begin{center}
\includegraphics{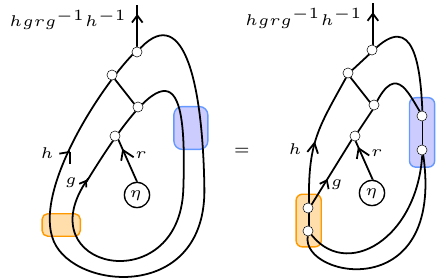}
\end{center}

Applying \ref{cor:73iv} twice, we have

\begin{center}
\includegraphics{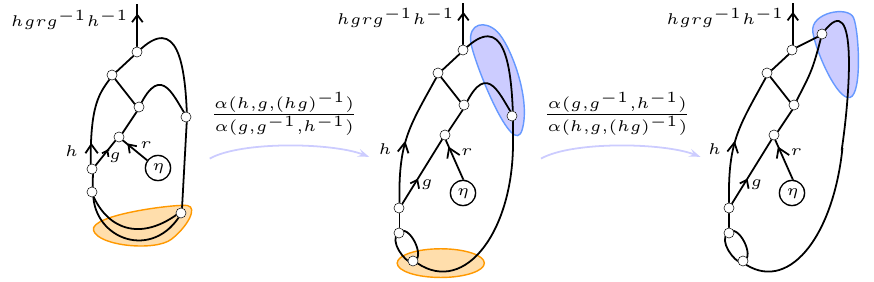}
\end{center}

These two factors cancel, so the first and last diagram in this figure
are equal.  We redraw this diagram by removing the loop (as per
\ref{cor:gfxm} (d)), then apply \ref{cor:gfxm} (c) to get 
\begin{center}
\includegraphics{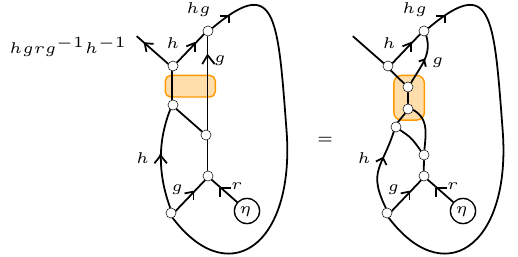}
\end{center}

Next, we apply \ref{def:proj2rep} (i) to obtain

\begin{center}
\includegraphics{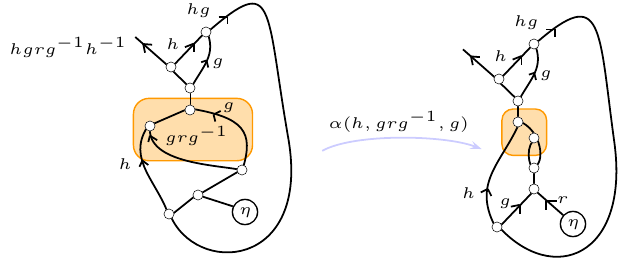}
\end{center}

By removing the loop and applying \ref{cor:rstringa}, we get
\begin{center}
\includegraphics{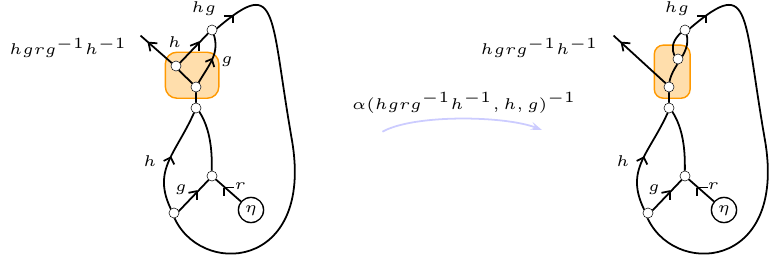}
\end{center}

Finally, we remove this loop then apply \ref{def:proj2rep} (i) to compute
\begin{center}
\includegraphics{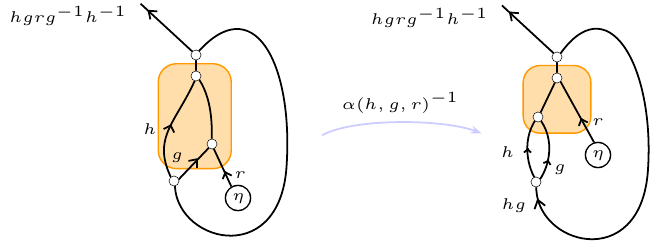}
\end{center}

After removing the loop we recognise this final diagram as representing
$\beta_{r,hg}(\eta)$.  We have therefore shown that 
\[
\frac{\alpha(h,grg^{-1},g)}{\alpha(hgrg^{-1}h^{-1},h,g)\cdot \alpha(h,g,r)}\cdot
\beta_{grg^{-1},h} \circ \beta_{r,g}(\eta) = \beta_{r,hg}(\eta), 
\]
as required.
\end{proof}

The main result of \cite{Willerton:Twisted_Drinfeld_double} identifies the twisted Drinfeld module
of $G$ for $\alpha$ with the twisted groupoid algebra
  \[
    D^\alpha(G) \cong \C^{\tau(\alpha)}[\Lambda G].
  \]
	and so we get the following corollary.
	
	\begin{corollary}[{compare \cite[Theorem 5.8]{KF:DDT}}]
	Let $G$ be a finite group, $\alpha$ a 3-cocycle on $G$ with values in $\C^\times$, and $\uG$ the corresponding
	categorical group.  Then representations of the inertia groupoid $\twotrunc \Lambda \uG$ are modules over the twisted Drinfeld double $D^\alpha(G)$.
	\label{cor:tigtwdd}
	\end{corollary}	

\section{Module categories and induction}
\subsection{Projective 2-representations as module
  categories}\label{sec:2-representations_as_module_categories} 
Let $\C$ be a field, and let 
\[
  \theta : G \times G \longrightarrow \C^{\times}
\] 
be a 2-cocycle.  Then there is an equivalence of
categories 
\begin{equation}
\mathrm{Rep}_\C^{\theta}(G) \simeq \C^{\theta}[G]\mathrm{-mod}
\label{eq:1drepequiv}
\end{equation}

from the projective $G$-representations with cocycle $\theta$ to
modules over the twisted group algebra $\C^{\theta}[G]$.  In the
context of 2-representations, $\C$ is replaced by the one-dimensional
2-vector space $\vctk$.  The \emph{categorified twisted group algebra}
$\vctk^{\alpha}[G]$ 
is the monoidal category of $G$-graded finite dimensional $\C$-vector
spaces, where the monoidal structure consists of the graded tensor
product, with 
associators twisted by $\alpha$ (see \cite{Ostrik:Module_categories_Drinfeld}, where
$\vctk^{\alpha}[G]$ is denoted $\vgw$). 

\begin{definition}
Let $\uC$ be a strict $\C$-linear 2-category, and let $\uG$ be the
skeletal 2-group classified by the (normalised) 3-cocycle $\alpha : G
\times G \times G \longrightarrow \C^{\times}$, as in Example
\ref{ex:ske2grp}.  
We write 
\begin{equation}
\mathrm{2Rep}_{\uC}^{\alpha}(G) := \mathrm{Bicat}( \bullet /\!\!/ \uG, \uC)
\end{equation}
for the 2-category of 2-representations as in \cite[Definition
7.1]{Bartlett:Unitary_2_representations}.\footnote{This is not the same as the category
  $\Hom_{\mathrm{2}\mbox{-}\mathrm{Grp}}(\uG, \mathrm{Aut}(\uC))$ in
  \cite{Frenkel:Gerbal_representations} after Definition 2.6.} 
\end{definition}

In the case where $\uC$ is the 2-category of finite dimensional
Kapranov-Voevodsky 2-vector spaces\footnote{A 2-vector space is a
  semisimple $\vctk$-module 
category with finitely many simple objects, see
\cite{Kapranov:Voevodsky}}, we will use the notation 
$\mathrm{2Rep}_{\vctk}^{\alpha}(G)$.  The 2-categorical analogue of
Equation \ref{eq:1drepequiv} is then 
\[
\mathrm{2Rep}_{\vctk}^{\alpha}(G) \simeq \vctk^\alpha[G]\mathrm{-mod}
\] 

We will switch freely between the points of view of module categories
and projective 2-representations. 

\begin{example}
Let $\theta$ be a 2-cochain on $G$ with boundary $d\theta = \alpha$.  Then
\[
\C^{\theta}[G] = \bigoplus_{g \in G} \C
\]
with multiplication twisted by $\theta$ is an algebra object in
$\vctk^{\alpha}[G]$.  Note that is it \emph{not} an algebra.  The
$\vctk^\alpha[G]$-module category 
\[
\vctk^\alpha[G]-\C^{\theta}[G]
\]
of right $\C^{\theta}[G]$-modules in $\vctk^\alpha[G]$ is the basic
example of a module category in \cite[\S 3.1]{Ostrik:Module_categories_Hopf}.  It translates
into our Example \ref{ex:trivp2rep} via 
the equivalence
\begin{eqnarray*}
F : \vctk^\alpha[G]-\C^{\theta}[G] &\longrightarrow& \vctk\\
\bigoplus_{g \in G} M_g &\longmapsto& M_1.
\end{eqnarray*}

Indeed, if we equip $\vctk$ with the module structure of Example
\ref{ex:trivp2rep}, then $F$ can be made a module functor as follows: 
given a $\C^{\theta}[G]$-module object $M$ in $\vctk^{\alpha}$ with
action 
\[
  M \otimes \C^{\theta}[G]
  \overset{s}{\underset{\cong}\longrightarrow} M,
\] \
we choose the isomorphism 
\[
M_{g^{-1}} = F(\C_g \otimes M) \longrightarrow \C_g \cdot F(M) = M_1
\]
to be the map
\[
M_{g^{-1}} \otimes \C_g \xhookrightarrow{\phantom{ABI}} F(M \otimes
\C^{\theta}[G]) \xrightarrow{F(s)} F(M). 
\]
\label{ex:babyequiv}
\end{example}

\subsection{Induced 2-representations}
Let $H \subset G$ be finite groups, and let $\alpha$ be a normalised
3-cocycle on $G$.  Let $\varrho$ be a projective 2-representation of
$H$ on $W \in \ob(\uC)$ 
with cocycle $\alpha|_{H}$.  

\begin{definition}
The induced 2-representation of $W$, if it exists, is characterised by
the universal property of a left-adjoint.  More precisely, an object
$\ind_H^G W$, together with 
a projective 2-representation $\ind_H^G \varrho$ with cocycle $\alpha$
and a 1-morphism 
\[
j : \varrho \longrightarrow \ind_H^G \varrho
\]
in $\mathrm{2Rep}_{\vctk}^{\alpha|_H}(H)$ is called {\em induced by
$\varrho$,} if for any projective $G$-2-representation $\pi$ on $V \in
\ob(\uC)$ with cocycle $\alpha$ and any 1-morphism of 
projective $H$-2-representations (for $\alpha|_H$)
\[
F : \varrho \longrightarrow \pi
\]
there exists a 1-morphism of projective $G$-2-representations (for $\alpha$)
\[
\bar{F} : \ind_H^G \varrho \longrightarrow \pi
\]
and a 2-isomorphism $\Phi$ fitting in the commuting diagram of $H$-maps

\begin{center}
\includegraphics{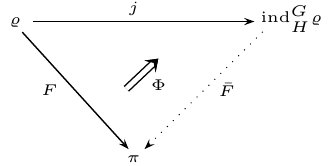}
\label{dia:hmaps}
\end{center}

such that $(\bar{F}, \Phi)$ is determined uniquely up to unique
2-isomorphism.  Here ``unique'' means that given two such pairs $(\bar{F_1},
\Phi_1)$ and $(\bar{F_2}, \Phi_2)$, there is a unique 2-isomorphism 
\[
\eta : \bar{F_1} \stackrel{\cong}{\Longrightarrow} \bar{F_2}
\]
satisfying
\[
(\eta j) \circ \Phi_1 = \Phi_2.
\]

If it exists, the induced projective 2-representation of $W$ is
determined uniquely up to a 1-equivalence in
$\mathrm{2Rep}^\alpha_{\vctk}(G)$, which is unique up to canonical
2-isomorphism.\end{definition} 

In the following, we abbreviate $\alpha|_H$ with $\alpha$.

\begin{proposition}
Let $\uM$ be a $\C$-linear left $\vctk^{\alpha}[H]$-module category.
Then the induced projective 2-representation of $\uM$ exists, 
and is given by the tensor product of $\vctk^{\alpha}[H]$-module categories
\[
\ind_H^G \uM = \vctk^{\alpha}[G] \boxtimes_{\vctk^{\alpha}[H]} \uM
\]
defined in \cite[Definition 3.3]{Etingof:Fusion_categories_and_homotopy_theory}.
\end{proposition}

\begin{proof}
Using the universal property of $- \boxtimes_{\vctk^\alpha[H]} -$, one
equips $\vctk^{\alpha}[G] \boxtimes_{\vctk^\alpha[H]} \uM$ with the
structure 
of a left $\vctk^{\alpha}[G]$-module category.  Using the universal
property of $- \boxtimes_{\vctk^\alpha[H]} -$ again, we deduce the
universal property for $\ind_H^G \uM$. 
\end{proof}

\begin{proposition}
Let $A$ be an algebra object in $\vctk^\alpha[H]$, and let $\uM =
\vctk^\alpha[H]\mathrm{-}A$ be the category of right $A$-module
objects in $\vctk^\alpha[H]$.  Then we have 
\[
\ind_H^G \uM = \vctk^\alpha[G]\mathrm{-}A,
\]
and the map $j$ is the canonical inclusion
\[
j : \vctk^\alpha[H]\mathrm{-}A \longrightarrow \vctk^\alpha[G]\mathrm{-}A
\]
i.e. $j(M)|_{H} = M$ and $j(M)|_{g} = 0$ for $g \notin H$.
\label{prop:indalg}
\end{proposition}
\begin{proof}
Let $M = \bigoplus_{g \in G} M_g$ be a right $A$-module object in
$\vctk^\alpha[G]$.  Then $M$ is the direct sum \emph{of $A$-module
  objects} 
\[
M = \bigoplus_{G/H} M|_{rH}
\]
where
\[
(M|_{rH})_s = \begin{cases}\begin{matrix} M_s & s \in rH\\ 0 &
    \mbox{otherwise}\end{matrix}\end{cases} 
\]
Fix a system $\uR$ of left coset representatives, and assume we are
given a $\vctk^{\alpha}[G]$-module category $\uN$ together with a
$\vctk^{\alpha}[H]$-module functor  
\[
F : \vctk^\alpha[H]\mathrm{-}A \longrightarrow \uN.
\]
We define the $\vctk^\alpha[G]$-module functor
\begin{align*}
\bar{F}\ :\ &\vctk^\alpha[G]\mathrm{-}A \longrightarrow \uN\\
&M \longmapsto \bigoplus_{r \in \uR} \C_r \cdot F(\C_{r^{-1}} \cdot M|_{rH}) 
\end{align*}
Then $\Phi$ is the inclusion of the summand $F(M)$ in $\bar{F}(j(M))$.
This $\Phi$ is an isomorphism, because the other summands of $\bar{F}
j(M)$ are (canonically) zero. 

Let $(\bar{F_2}, \Phi_2)$ be a second pair fitting into the diagram on
page \pageref{dia:hmaps}, for instance, from a different choice of
coset representatives.  Then the isomorphism 
$\eta : \bar{F} \Longrightarrow \bar{F_2}$ is the inverse of the composition
\[
\bar{F_2}(M) \cong \bigoplus_{r \in \uR} \bar{F_2}(M|_{rH}) \cong
\bigoplus_{r \in \uR} \C_r \cdot \bar{F_2}(\C_{r^{-1}} \otimes
M|_{rH}) \stackrel{\Phi_2}{\cong} \bar{F}(M)  
\]
\end{proof}

\begin{pcorollary*}
As right $\vctk^\alpha[H]$-modules
\[
\vctk^\alpha[G] \cong \bigoplus_{\stackrel{r \in \uR}{G/H}} \vctk^\alpha[H]
\]
\end{pcorollary*}

\subsection{Comparison of classifications}
In \cite[Proposition 7.3]{GK:Representation_and_character_theory}, the finite dimensional
2-re\-pre\-sen\-ta\-tions are classified.  In \cite[Example 2.1]{Ostrik:Module_categories_Drinfeld},
the indecomposable module categories 
over $\vctk^{\alpha}[G]$ are classified.  In \cite{GK:Representation_and_character_theory}, a
comparison with Ostrik's work was attempted, but the dictionary
established in the
previous section appears to be more suitable, as it translates directly between these
two results. Indeed, for trivial $\alpha$, the following corollary
specialises to \cite[Proposition 7.3]{GK:Representation_and_character_theory}.
\begin{corollary}
Let $\varrho$ be a projective 2-representation of a group $G$ with
3-cocycle $\alpha$ on a semisimple 
$\C$-linear abelian category $\uV$ with finitely many simple objects. Then
\[
\uV \cong \bigoplus_{i=1}^m \ind_{H_i}^G \varrho_{\theta_i}
\]
where the $H_i$ are subgroups of $G$, $\theta_i$ is a 2-cochain on
$H_i$ such that $d\theta_i = \alpha|_{H_i}$,  
and $\varrho_{\theta_i}$ is the projective 2-representation
corresponding to the pair $(H_i,\theta_i)$ described 
in Examples \ref{ex:trivp2rep} and \ref{ex:babyequiv}.
\end{corollary}

\begin{proof}
Ostrik's result \cite[Example 2.1]{Ostrik:Module_categories_Drinfeld} yields a decomposition
\begin{align*}
\uV &\simeq \bigoplus_{i=1}^m \vctk^\alpha[G]-\C^{\theta_i}[H_i]\\
	  &\simeq \bigoplus_{i=1}^m \ind_{H_i}^{G} (\vctk^{\alpha}[H_i] - \C^{\theta_i}[H_i])\\
		&\simeq \bigoplus_{i=1}^m \ind_{H_i}^G \varrho_{\theta_i}
\end{align*}
Here, the second equivalence is Proposition \ref{prop:indalg}, the
third is Example \ref{ex:babyequiv}, and $\varrho_{\theta_i}$ is 
as in Example \ref{ex:trivp2rep}.
\end{proof}

\bibliographystyle{alphaurl}
\bibliography{mylib}

\end{document}